\documentclass[a4paper,10pt]{amsart}

\usepackage[utf8]{inputenc} 
\usepackage{amsmath}
\usepackage{amsbsy}
\usepackage{amssymb}
\usepackage{amscd,amsthm}
\usepackage{verbatim}
\usepackage[english]{babel}
\usepackage{dsfont}
\usepackage{anysize}
\usepackage{hyperref}

\newcommand{\Z}{\mathds{Z}}
\newcommand{\Q}{\mathds{Q}}
\newcommand{\N}{\mathds{N}}
\newcommand{\R}{\mathds{R}}

\newcommand{\p}{\phantom}
\newcommand{\q}{\quad}

\newtheorem{thm}{Theorem}[section]
\newtheorem{lem}[thm]{Lemma}
\newtheorem{kor}[thm]{Corollary}

\newtheorem{prop}[thm]{Proposition}

\theoremstyle{definition}
\newtheorem*{defi}{Definition}
\newtheorem*{nota}{Notation}

\theoremstyle{remark}
\newtheorem*{rema}{Remark}

\title{New estimates for the $n$th prime number}
\author{Christian Axler}
\address{Institute of Mathematics\\ Heinrich-Heine University Düsseldorf \\ 40225 Düsseldorf, Germany}
\email{christian.axler@hhu.de}
\date{\today}

\begin{document}

\begin{abstract}
In this paper we establish new upper and lower bounds for the $n$th prime number, which improve several existing bounds of similar shape. As the main tool we 
use some recently obtained explicit estimates for the prime counting function. A further main tool is the use of estimates concerning the reciprocal of $\log 
p_n$. As an application we derive new estimates for $\vartheta(p_n)$, where $\vartheta(x)$ is the Chebyshev's $\vartheta$-function.
\end{abstract}

\maketitle

\section{Introduction}

Let $p_n$ denote the $n$th prime number and let $\pi(x)$ be the number of primes not exceeding $x$. In 1896, Hadamard \cite{hadamard1896} and de la 
Vall\'{e}e-Poussin \cite{vallee1896} independently proved the asymptotic formula $\pi(x) \sim x/\log x$ as $x \to \infty$, which is known as the \textit{Prime 
Number Theorem}. (Here $\log x$ is the natural logarithm of $x$.) As a consequence of the Prime Number Theorem, one gets the asymptotic expression 
\begin{equation}
p_n \sim n \log n \q\q (n \to \infty). \tag{1.1} \label{1.1}
\end{equation}
Cipolla \cite{cp} found a more precise result. He showed that for every positive integer $m$ there exist unique monic polynomials $T_1, \ldots, T_m$ with 
rational coefficients and $\deg(T_k)= k$ so that
\begin{equation}
p_n = n \left( \log n + \log \log n - 1 + \sum_{k=1}^m \frac{(-1)^{k+1}T_k(\log \log n)}{k\log^kn} \right) 
+ O\left( \frac{n (\log \log n)^{m+1}}{\log^{m+1}n} \right). \tag{1.2} \label{1.2}
\end{equation}
The polynomials $T_k$ can be computed explicitly. In particular, $T_1(x) = x - 2$ and $T_2(x) = x^2 - 6x + 11$ (see Cipolla \cite{cp} or Salvy \cite{salvy} for 
further details). Since the computation of the $n$th prime number is difficult for large $n$, we are interested in explicit estimates for $p_n$. The 
asymptotic formula \eqref{1.2} yields
\begin{align}
& p_n > n \log n, \tag{1.3} \label{1.3} \\
& p_n < n ( \log n + \log \log n), \tag{1.4} \label{1.4} \\
& p_n > n ( \log n + \log \log n - 1) \tag{1.5} \label{1.5}
\end{align}
for all sufficiently large values of $n$. The first result concerning a lower bound for the $n$th prime number is due to Rosser \cite[Theorem 1]{jr}. He 
showed that the inequality \eqref{1.3} holds for every positive integer $n$. In the literature, this result is often called \textit{Rosser's theorem}. 
Moreover, he proved \cite[Theorem 2]{jr} that
\begin{equation}
p_n < n(\log n + 2 \log \log n) \tag{1.6} \label{1.6}
\end{equation}
for every $n \geq 4$. The next results concerning the upper and lower bounds that correspond to the first three terms of \eqref{1.2} are due to Rosser and 
Schoenfeld \cite[Theorem 3]{rosserschoenfeld1962}. They refined Rosser's theorem and the inequality \eqref{1.6} by showing that
\begin{displaymath}
p_{n} > n(\log n + \log \log n - 1.5) 
\end{displaymath}
for every $n \geq 2$ and that the inequality
\begin{equation}
p_{n} < n(\log n + \log \log n - 0.5) \tag{1.7} \label{1.7}
\end{equation}
holds for every $n \geq 20$. The inequality \eqref{1.7} implies that \eqref{1.4} is fulfilled for every $n \geq 6$. Based on their estimates for the Chebyshev 
functions $\psi(x)$ and $\vartheta(x)$, Rosser and Schoenfeld \cite{rosserschoenfeld1975} announced to have new estimates for the $n$th prime number $p_n$ but 
they have never published the details. In the direction of \eqref{1.5}, Robin \cite[Lemme 3, Th\'{e}or\`{e}me 8]{robin1983} showed that
\begin{equation}
p_{n} \geq n(\log n + \log \log n - 1.0072629) \tag{1.8} \label{1.8}
\end{equation}
for every $n \geq 2$, and that the inequality \eqref{1.5} holds for every integer $n$ such that $2 \leq n \leq \pi(10^{11})$. Massias and Robin 
\cite[Th\'{e}or\`{e}me A]{mr} gave a series of improvements of \eqref{1.7} and \eqref{1.8}. For instance, they have found that $p_{n} \geq n(\log n + \log \log 
n - 1.002872)$ for every $n \geq 2$.
%
In his thesis, Dusart showed that the inequality 
\begin{equation}
p_n \leq n \left( \log n + \log \log n - 1 + \frac{\log \log n - 1.8}{\log n} \right) \tag{1.9} \label{1.9}
\end{equation}
holds for every $n \geq 27\,076$, see \cite[Th\'{e}or\`{e}me 1.7]{pd}. Further, he made a breakthrough concerning the inequality \eqref{1.5} by showing that 
this inequality holds for every $n \geq 2$.
The current best estimates for the $n$th prime, which correspond to the first terms in \eqref{1.2}, are also given by Dusart \cite[Propositions 5.15 and 
5.16]{dusart2017}. He used explicit estimates for Chebyshev's $\vartheta$-function to show that the inequality
\begin{equation}
p_n \leq n \left( \log n + \log \log n - 1 + \frac{\log \log n - 2}{\log n} \right), \tag{1.10} \label{1.10}
\end{equation}
which corresponds to the first four terms of \eqref{1.2}, holds for every $n \geq 688\,383$ and that
\begin{equation}
p_n \geq n \left( \log n + \log \log n - 1 + \frac{\log \log n - 2.1}{\log n} \right) \tag{1.11} \label{1.11}
\end{equation}
for every $n \geq 3$. The goal of this paper is to improve the inequalities \eqref{1.10} and \eqref{1.11} with regard to Cipolla's asymptotic expansion 
\eqref{1.2}. For this purpose, we use estimates for the quantity $1/\log p_n$ and some estimates \cite{ca2017} for the prime counting function $\pi(x)$ to 
%
obtain the following refinement of \eqref{1.10}.

\begin{thm} \label{thm101}
For every integer $n \geq 46\,254\,381$, we have
\begin{equation}
p_n < n \left( \log n + \log \log n - 1 + \frac{\log \log n - 2}{\log n} - \frac{(\log \log n)^2-6\log \log n + 10.667}{2 \log^2 n} \right). \tag{1.12} 
\label{1.12}
\end{equation}
\end{thm}

Furthermore, we
%
%
%
give the following improvement of \eqref{1.11}.

\begin{thm} \label{thm102}
For every integer $n \geq 2$, we have
\begin{equation}
p_n > n \left( \log n + \log \log n - 1 + \frac{\log \log n - 2}{\log n} - \frac{(\log\log n)^2-6\log \log n + 11.508}{2\log^2 n} \right). \tag{1.13} 
\label{1.13}
\end{equation}
\end{thm}

As an application, we find some refined estimates for $\vartheta(p_n)$, where $\vartheta(x) = \sum_{p \leq x} \log p$ is Chebyshev's $\vartheta$-function.

\begin{nota}
Throughout this paper, let $n$ denote a positive integer. For better readability, in the majority of the proofs we use the notation
\begin{displaymath}
w = \log \log n, \q y = \log n, \q z = \log p_n.
\end{displaymath}
\end{nota}

\section{Effective estimates for the reciprocal of $\log p_n$}

In order to prove Theorems \ref{thm101} and \ref{thm102}, we need sharp estimates for the quantity $1/\log p_n$. Cipolla \cite[p.\:139]{cp} showed that
\begin{displaymath}
\frac{1}{\log p_n} = \frac{1}{\log n} - \frac{\log \log n}{\log^2n} + o \left( \frac{1}{\log^2 n} \right).
\end{displaymath}
Concerning this asymptotic formula, we give the following inequality involving $1/\log p_n$, where the polynomials $P_1, \ldots, P_4 \in \Z[x]$ are given by
\begin{itemize}
\item $P_1(x) = 3x^2 - 6x + 5$,
\item $P_2(x) = 5x^3 - 24x^2 + 39x - 14$,
\item $P_3(x) = 7x^4 - 48x^3 + 120x^2 - 124x + 51$,
\item $P_4(x) = 9x^5 - 80x^4 + 280x^3 - 480x^2 + 405x - 124$,
\end{itemize}

\begin{prop} \label{prop204}
For every integer $n \geq 688\,383$, we have
\begin{displaymath}
\frac{1}{\log p_n} \geq \frac{1}{\log n} - \frac{\log \log n}{\log^2n} + \frac{(\log \log n)^2 - \log \log n + 1}{\log^2n\log p_n} + \frac{1}{\log p_n} 
\sum_{k=1}^4 \frac{(-1)^{k+1}P_k(\log \log n)}{k(k+1) \log^{k+2} n}.
\end{displaymath}
\end{prop}

\begin{proof}
Let $n$ be a integer with $n \geq 688\,383$. For convenience, we write $w = \log \log n$, $y=\log n$, and $z = \log p_n$.
Applying the inequality $\log(1+x) \leq \sum_{k=1}^7 (-1)^{k+1} x^k/k$, which holds for every $x > -1$, and the fact that $(w-1)/y + (w-2)/y^2 > -1$ 
to \eqref{1.10}, we see that
\begin{displaymath}
- y^2 + (y - w) z \leq - w^2 + (y - w) \sum_{k=1}^7 \frac{(-1)^{k+1}}{k} \left( \frac{w - 1}{y} + \frac{w - 2}{y^2} \right)^k.
\end{displaymath}
Extending the right-hand side of the last inequality, we obtain the inequality
\begin{align}
- y^2 + (y - w) z & < - w^2 + w - 1 - \frac{P_1(w)}{2y} + \frac{P_2(w)}{6y^2} - \frac{P_3(w)}{12y^3} + \frac{P_4(w)}{20y^4} - \frac{P_5(w)}{30y^5} 
+ \frac{P_6(w)}{42y^6} \tag{2.5} \label{2.5}\\
& \p{\q\q} - \frac{P_7(w)}{28y^7} - \frac{(w-2)Q_1(w)}{12y^8} - \frac{(w-2)^2Q_2(w)}{30y^9} - \frac{(w-2)^3Q_3(w)}{10y^{10}} \nonumber \\
& \p{\q\q} - \frac{(w-2)^4Q_4(w)}{6y^{11}} - \frac{(w-2)^5Q_5(w)}{6y^{12}} - \frac{(w-2)^6Q_6(w)}{7y^{13}} - \frac{w(w-2)^7}{7y^{14}}, \nonumber
\end{align}
where the polynomials $P_5, P_6, P_7, Q_1, Q_2, Q_3, Q_4, Q_5, Q_6 \in \Z[x]$ are given by
\begin{itemize}
\item $P_5(x) = 11x^6 -120x^5 + 540x^4 - 1280x^3+1680x^2 - 1146x + 325$,
\item $P_6(x) = 13x^7 - 168x^6 + 924x^5 - 2800x^4 + 5040x^3 - 5376x^2 + 3143x - 762$,
\item $P_7(x) = 4x^8 - 84x^7 + 630x^6 - 2492x^5 + 5915x^4 - 8764x^3 + 7966x^2 - 4064x + 896$,
\item $Q_1(x) = 12x^7 - 138x^6 + 676x^5-1819x^4+2914x^3 - 2782x^2 + 1468x - 328$,
\item $Q_2(x) = 90x^6 - 700x^5 + 2405x^4 - 4506x^3 + 4801x^2 - 2732x + 648$,
\item $Q_3(x) = 50x^5-275x^4+662x^3-833x^2+538x-140$,
\item $Q_4(x) = 30x^4-114x^3+181x^2-136x+40$,
\item $Q_5(x) = 18x^3-43x^2+38x-12$,
\item $Q_6(x) = 7x^2 - 8x + 2$.
\end{itemize}
For $x \geq 2$, we have $Q_i(x) \geq 0$, $1 \leq i \leq 6$, and $x(x-2)^7 \geq 0$. Combined with \eqref{2.5}, it gives
\begin{displaymath}
- y^2 + (y - w) z < - w^2 + w - 1 - \frac{P_1(w)}{2y} + \frac{P_2(w)}{6y^2} - \frac{P_3(w)}{12y^3} + \frac{P_4(w)}{20y^4} - \frac{P_5(w)}{30y^5} 
+ \frac{P_6(w)}{42y^6} - \frac{P_7(w)}{28y^7}.
\end{displaymath}
By \cite[Lemma 2.3]{axler2013}, we have $P_5(w)/30 - P_6(w)/(42y) + P_7(w)/(28y^2) \geq 0$ which completes the proof.
\end{proof}


\begin{kor} \label{kor205}
For every integer $n \geq 456\,914$, we have
\begin{displaymath}
\frac{1}{\log p_n} \geq \frac{1}{\log n} - \frac{\log \log n}{\log^2n} + \frac{(\log \log n)^2 - \log \log n + 1}{\log^2n\log p_n} 
+ \frac{P_1(\log \log n)}{2\log^3n \log p_n} - \frac{P_2(\log \log n)}{6\log^4n \log p_n}.
\end{displaymath}
\end{kor}

\begin{proof}
See \cite[Korollar 2.6]{axler2013}.
\end{proof}

\begin{kor} \label{kor206}
For every integer $n \geq 71$, we have
\begin{displaymath}
\frac{1}{\log p_n} \geq \frac{1}{\log n} - \frac{\log \log n}{\log^2n} + \frac{(\log \log n)^2 - \log \log n + 1}{\log^2n\log p_n}.
\end{displaymath}
\end{kor}

\begin{proof}
Since the inequality
\begin{equation}
\frac{P_1(\log \log n)}{2\log n} - \frac{P_2(\log \log n)}{6\log^2 n} \geq 0 \tag{2.6} \label{2.6}
\end{equation}
holds for every $n \geq 3$, Corollary \ref{kor205} implies the validity of the required inequality for every $n \geq 456\,914$. We finish by checking the 
remaining cases with a computer.
\end{proof}

Using a similar method as in the proof of Proposition \ref{prop204}, we find the following inequality involving the reciprocal of $\log p_n$.


\begin{prop} \label{prop303}
For every integer $n \geq 2$, we have
\begin{displaymath}
\frac{1}{\log p_n} \leq \frac{1}{\log n} - \frac{\log \log n}{\log^2 n} + \frac{(\log \log n)^2 - \log \log n + 1}{\log^2n \log p_n} 
+ \frac{P_8(\log \log n)}{2\log^3 n \log p_n} - \sum_{k=4}^6 \frac{P_{k+5}(\log \log n)}{2\log^k n \log p_n}.
\end{displaymath}
\end{prop}

\begin{proof}
First, we consider the case where $n \geq 33$. For convenience, we write again $w = \log \log n$, $y=\log n$, and $z = \log p_n$. Notice that $\log(1+t) \geq t 
- t^2/2$ for every $t \geq 0$. If we combine the last fact with \eqref{1.11} and $(w-1)/y + (w-2.1)/y^2 \geq 0$, we obtain the inequality
\begin{displaymath}
- y^2 + (y - w)z \geq - w^2 + (y - w) \sum_{k=1}^2 \frac{(-1)^{k+1}}{k} \left( \frac{w - 1}{y} + \frac{w - 2.1}{y^2} \right)^k
\end{displaymath}
which implies the required inequality. A computer check completes the proof.
\end{proof}

Proposition \ref{prop303} implies the following two corollaries. The proofs are left to the reader.

\begin{kor} \label{kor304}
For every integer $n \geq 2$, we have
\begin{displaymath}
\frac{1}{\log p_n} \leq \frac{1}{\log n} - \frac{\log \log n}{\log^2 n} + \frac{(\log \log n)^2 - \log \log n + 1}{\log^2n \log p_n} + \frac{P_8(\log \log 
n)}{2\log^3 n \log p_n} - \sum_{k=4}^5 \frac{P_{k+5}(\log \log n)}{2\log^k n \log p_n}.
\end{displaymath}
\end{kor}


\begin{kor} \label{kor305}
For every integer $n \geq 2$, we have
\begin{displaymath}
\frac{1}{\log p_n} \leq \frac{1}{\log n} - \frac{\log \log n}{\log^2 n} + \frac{(\log \log n)^2 - \log \log n + 1}{\log^2n \log p_n}
+ \frac{P_8(\log \log n)}{2\log^3 n \log p_n} - \frac{P_9(\log \log n)}{2\log^4 n \log p_n}.
\end{displaymath}
\end{kor}


\section{Proof of Theorem \ref{thm101}}

First, we introduce the following notation. Let the polynomials $P_1, \ldots, P_4 \in \Z[x]$ are given as in the beginning of Section 2. Let $A_0$ be a real 
number with $0.75 \leq A_0 < 1$ and let $F_0 : \N \rightarrow \R$ be defined by
\begin{displaymath}
F_0(n) = \log n - A_0\log p_n.
\end{displaymath}
From \eqref{1.1}, it follows that $F_0(n)$ is nonnegative for all sufficiently large values of $n$. So we can define
\begin{equation}
N_0 = N(A_0) = \min \{ k \in \N \mid F_0(n) \geq 0 \;\; \text{for every $n \geq k$} \}. \tag{2.1} \label{2.1}
\end{equation}
Let $A_1$ be a real number with $0 < A_1 \leq 458.7275$, and for $w = \log \log n$ let $F_1 : \N_{\geq 2} \rightarrow \R$ be given by
\begin{align*}
F_1(n) & = \frac{A_1}{\log^5p_n} + \frac{(w^2 - 3.85w + 14.15)(w^2 - w + 1)}{\log^4 n\log p_n} + \frac{2.85P_1(w)}{2\log^3 n \log^2 p_n} + \frac{2.85P_1(w)}{2 
\log^4 n \log p_n} \\
& \p{\q\q} + \left( \frac{13.15(w^2 - w + 1)}{\log^2 n \log^2 p_n} - \frac{70.7w}{\log^2 n \log^2 p_n}\right) \left( \frac{1}{\log n} + \frac{1}{\log p_n} 
\right) - \frac{P_2(w)}{6 \log^4 n \log p_n}.
\end{align*}
Then $F_1(n)$ is nonnegative for all sufficiently large values of $n$, and we can define
\begin{equation}
N_1 = N_1(A_1) = \min \{ k \in \N \mid F_1(n) \geq 0 \;\; \text{for every $n \geq k$}\}. \tag{2.2} \label{2.2}
\end{equation}
Further we set $A_2 = (458.7275-A_1)A_0^5$ and $A_3 = 3428.7225A_0^6$.
To prove Theorem \ref{thm101}, we first use a recently obtained estimate \cite{ca2017} for the prime counting function $\pi(x)$ to construct a positive integer 
$n_0$ and an arithmetic function $b_0 : \N_{\geq 2} \to \R$, both depending on some parameters, with 
$b_0(n) \to 10.7$ as $n \to \infty$ so that
\begin{displaymath}
p_n < n \left( \log n + \log \log n - 1 + \frac{\log \log n - 2}{\log n} - \frac{(\log \log n)^2-6\log \log n + b_0(n)}{2 \log^2 n} \right)
\end{displaymath}
for every $n \geq n_0$. In order to do this, let $a_0 : \N_{\geq 2} \rightarrow \R$ be an arithmetic function satisfying
\begin{equation}
a_0(n) \geq - (\log \log n)^2 + 6 \log \log n, \tag{2.7} \label{2.7}
\end{equation}
and let $N_2$ be a positive integer depending on the arithmetic function $a_0$ so that the inequalities
\begin{equation}
-1 < \frac{\log \log n-1}{\log n} + \frac{\log \log n-2}{\log^2 n} - \frac{(\log \log n)^2 - 6\log \log n + a_0(n)}{2\log^3 n} \leq 1, \tag{2.8} \label{2.8}
\end{equation}
\begin{equation}
\frac{\log \log n - 2}{\log^2 n} - \frac{(\log \log n)^2 - 6\log \log n + a_0(n)}{2\log^3 n} \geq 0, \q \text{and} \tag{2.9} \label{2.9}
\end{equation}
\begin{equation}
p_n < n \left( \log n + \log \log n - 1 + \frac{\log \log n - 2}{\log n} - \frac{(\log \log n)^2 - 6\log \log n + a_0(n)}{2 \log^2 n} \right) \tag{2.10} 
\label{2.10}
\end{equation}
hold simultaneously for every $n \geq N_2$. Now we set
\begin{align*}
G_0(x) & = \frac{2x^3 - 21x^2 + 82.2x - 98.9}{6e^{3x}} - \frac{x^4 - 14x^3 + 53.4x^2 - 100.6x + 17}{4e^{4x}} \\
& \p{\q\q} + \frac{2x^5 - 10x^4 + 35x^3 - 110x^2 + 150x - 42}{10e^{5x}} - \frac{3x^4 - 44x^3 + 156x^2 - 96x + 64}{24e^{6x}},
\end{align*}
and for $w = \log \log n$ we define
\begin{align}
b_0(n) & = 10.7 + \frac{2A_2}{\log^3 n} + \frac{2A_3}{\log^4 n} + \frac{a_0(n)}{\log n} \left( 1 - \frac{w- 1}{\log n} - \frac{w-2}{\log^2 n} + \frac{2w^2 - 
12w + a_0(n)}{4\log^3 n}\right) \tag{2.11} \label{2.11} \\
& \p{\q\q} - 2\,G_0(w)\log^2 n + \frac{A_0((5.7A_1+8.7)w^2 - (32A_0 + 38)w + 147.1A_0 + 10.7)}{\log^2 n} \nonumber \\
& \p{\q\q} + \frac{2 \cdot 70.7A_0^3(w^2-w+1)}{\log^4 n} + \frac{2 \cdot 70.7A_0^4(w^2-w+1)}{\log^4 n}. \nonumber
\end{align}
Then we obtain the following

\begin{prop} \label{prop207}
For every integer $n \geq \max \{ N_0, N_1, N_2, 841\,424\,976 \}$, we have
\begin{displaymath}
p_n < n \left( \log n + \log \log n - 1 + \frac{\log \log n - 2}{\log n} - \frac{(\log \log n)^2-6\log \log n + b_0(n)}{2 \log^2 n} \right).
\end{displaymath}
\end{prop}

In order to prove this proposition, we need the following lemma. The proof is left to the reader.

\begin{lem} \label{lem201}
For every $x \geq 2.11$, we have
\begin{equation}
\frac{(x^2 - 3.85x + 14.15)P_1(x)}{2} - \frac{2.85 P_2(x)}{3} + \frac{P_3(x)}{12} - \frac{(x^2 - 3.85x + 14.15)P_2(x)}{6e^x} - \frac{P_4(x)}{20e^x} \geq 0. 
\tag{2.3} \label{2.3}
\end{equation}
\end{lem}


Now we give a proof of Proposition \ref{prop207}.

\begin{proof}[Proof of Proposition \ref{prop207}]
Let $n \geq \max \{ N_0, N_1, N_2, 841\,424\,976 \}$. By \cite[Theorem 3.8]{ca2017}, we have
\begin{equation}
p_n < n \left( \log p_n - 1 - \frac{1}{\log p_n} - \frac{2.85}{\log^2p_n} - \frac{13.15}{\log^3p_n} - \frac{70.7}{\log^4p_n} - \frac{458.7275}{\log^5p_n} 
- \frac{3428.7225}{\log^6p_n} \right). \tag{2.12} \label{2.12}
\end{equation}
For convenience, we write $w = \log \log n$, $y = \log n$, and $z = \log p_n$.
By Corollary \ref{kor205}, we have
\begin{equation}
\frac{1}{z^2} \geq \frac{1}{yz} - \frac{w}{y^2z} + \frac{w^2-w+1}{y^2z^2} + \frac{P_1(w)}{2y^3z^2} - \frac{P_2(w)}{6y^4z^2}. \tag{2.13} \label{2.13}
\end{equation}
Again using Corollary \ref{kor205}, we get
\begin{equation}
\frac{1}{yz} \geq \Phi_1(n) = \frac{1}{y^2} - \frac{w}{y^3} + \frac{w^2-w+1}{y^3z} + \frac{P_1(w)}{2y^4z} - \frac{P_2(w)}{6y^5z}. \tag{2.14} \label{2.14}
\end{equation}
Applying \eqref{2.14} to \eqref{2.13}, we see that
\begin{equation}
\frac{1}{z^2} \geq \Phi_2(n) = \frac{1}{y^2} - \frac{w}{y^3} - \frac{w}{y^2z} + \frac{w^2-w+1}{y^3z} + \frac{w^2-w+1}{y^2z^2} + \left( \frac{P_1(w)}{2y^3z} - 
\frac{P_2(w)}{6y^4z} \right) \left( \frac{1}{y} + \frac{1}{z} \right). \tag{2.15} \label{2.15}
\end{equation}
Now \eqref{2.6} implies that
\begin{equation}
\frac{1}{z^2} \geq \Phi_3(n) = \frac{1}{y^2} - \frac{w}{y^3} - \frac{w}{y^2z} + \frac{w^2-w+1}{y^3z} + \frac{w^2-w+1}{y^2z^2}. \tag{2.16} \label{2.16}
\end{equation}
We assumed $n \geq N_0$. Hence $F_0(n) \geq 0$, which is equivalent to
\begin{equation}
\frac{A_0}{y} \leq \frac{1}{z}. \tag{2.17} \label{2.17}
\end{equation}
From \eqref{2.17} and the fact that $2.85x^2 - 16x + 73.55 \geq 0$ for every $x \geq 0$, it follows
\begin{equation}
\frac{2.85w^2-16w+73.55}{z^2} \geq \frac{A_0(5.7w^2 - 32w + 147.1)}{2yz}. \tag{2.18} \label{2.18}
\end{equation}
Let $f(x) = (5.7A_0+8.7)x^2 - (32A_0 + 38)x + 147.1A_0 + 10.7$. Since $0.75 \leq A_0 < 1$, we get $f(x) \geq 12.975x^2 - 70x + 121.025 \geq 0$ for every $x 
\geq 0$. Using \eqref{2.17} and \eqref{2.18}, we get
\begin{align}
& \frac{2.85w^2 - 16w + 73.55}{z^2} + \frac{8.7w^2 - 38w + 10.7}{2yz} \geq \frac{A_0f(w)}{2y^2}. \tag{2.19} \label{2.19}
\end{align}
We recall that $A_2 = (458.7275-A_1)A_0^5$ and $A_3 = 3428.7225A_0^6$. Hence \eqref{2.17} implies that
\begin{equation}
\frac{A_2}{y^5} + \frac{A_3}{y^6} + \frac{70.7A_0^3}{y^6} + \frac{70.7A_0^4}{y^6} \leq \frac{458.7275 - A_1}{z^5} + \frac{3428.7225}{z^6} + \frac{70.7}{y^3z^3} 
+ \frac{70.7}{y^2z^4}. \tag{2.20} \label{2.20}
\end{equation}
Now we apply \eqref{2.19} and \eqref{2.20} to \eqref{2.11} and see that
\begin{align}
\frac{10.7 - b_0(n)}{2y^2} & + \frac{2.85(w^2 - w + 1)}{y^2z^2} - \frac{13.15w}{y^2z^2} + \frac{70.7}{y^2z^2} + \frac{8.7w^2 - 38w + 10.7}{2y^3z} + 
\frac{458.7275}{z^5} \tag{2.21} \label{2.21} \\
& \p{\q\q} - \frac{A_1}{z^5} + \frac{3428.7225}{z^6} + \frac{70.7(w^2-w+1)}{y^2z^3} \left( \frac{1}{y} + \frac{1}{z} \right) \nonumber \\
& \geq G_0(w) - \frac{a_0(n)}{2y^3} \left( 1 - \frac{w- 1}{y} - \frac{w-2}{y^2} + \frac{2w^2 - 12w + a_0(n)}{4y^3}\right). \nonumber
\end{align}
The inequality \eqref{2.6} tells us that
\begin{equation}
\frac{13.15}{z} \left( \frac{P_1(w)}{2y^3z} - \frac{P_2(w)}{6y^4z} \right)\left( \frac{1}{y} + \frac{1}{z} \right) \geq 0. \tag{2.22} \label{2.22}
\end{equation}
Adding the left-hand side of \eqref{2.22} and the left-hand side of \eqref{2.3} with $x = w$ to the left-hand side of \eqref{2.21}, we get
\begin{align*}
\frac{5.35}{y^2} &- \frac{b_0(n)}{2y^2} + \frac{2.85(w^2 - w + 1)}{y^2z^2} - \frac{13.15w}{y^2z^2} + \frac{70.7}{y^2z^2} + \frac{8.7w^2 - 38w + 10.7}{2y^3z} + 
\frac{458.7275}{z^5} + \frac{3428.7225}{z^6} \\
& \p{\q\q} - \frac{A_1}{z^5} + \frac{70.7(w^2 - w + 1)}{y^2z^3} \left( \frac{1}{y} + \frac{1}{z} \right) + \frac{13.15}{z} \left( \frac{P_1(w)}{2y^3z} - 
\frac{P_2(w)}{6y^4z} \right)\left( \frac{1}{y} + \frac{1}{z} \right) - \frac{2.85 P_2(w)}{6 y^5z} \\
& \p{\q\q} - \frac{2.85P_2(w)}{6y^4z^2} + \frac{(w^2-3.85w+14.15)P_1(w)}{2 y^5z} - \frac{(w^2-3.85w+14.15)P_2(w)}{6y^6z} + \frac{P_3(w)}{12y^5z} - 
\frac{P_4(w)}{20y^6z} \\
& \geq G_0(w) - \frac{a_0(n)}{2y^3} \left( 1 - \frac{w- 1}{y} - \frac{w-2}{y^2} + \frac{2w^2 - 12w + a_0(n)}{4y^3}\right).
\end{align*}
Since $n \geq N_1$, we have $F_1(n) \geq 0$. Now we add $F_1(n)$ to the left-hand side of the last inequality, use the identity $8.7w^2 - 38w + 10.7 = P_1(w) + 
2\cdot 2.85(w^2-w+1) - 2 \cdot 13.15w$, and collect all terms containing the number $70.7$ and the term $w^2-3.85w+14.15$, respectively, to get
\begin{align*}
\frac{5.35}{y^2} &- \frac{b_0(n)}{2y^2} + \frac{2.85(w^2 - w + 1)}{y^2z^2} - \frac{13.15w}{y^2z^2} + \frac{70.7}{z^2} \cdot \Phi_3(n) + \frac{458.7275}{z^5} + 
\frac{3428.7225}{z^6} + \frac{2.85(w^2 - w + 1)}{y^3z} \\
& \p{\q\q} - \frac{13.15w}{y^3z} + \left( 2.85 + \frac{13.15}{z} \right) \left( \frac{P_1(w)}{2y^3z} - \frac{P_2(w)}{6y^4z} \right)\left( \frac{1}{y} + 
\frac{1}{z} \right) + \frac{w^2-3.85w+14.15}{y} \cdot \Phi_1(n) \\
& \p{\q\q} + \frac{P_1(w)}{2y^3z} - \frac{P_2(w)}{6 y^4z} + \frac{P_3(w)}{12y^5z} - \frac{P_4(w)}{20y^6z} + \frac{13.15(w^2 - w + 1)}{y^2z^2}\left( \frac{1}{y} 
+ \frac{1}{z} \right) - \frac{2.85w}{y^3} \\
& \geq \widetilde{G_0}(w) - \frac{a_0(n)}{2y^3} \left( 1 - \frac{w- 1}{y} - \frac{w-2}{y^2} + \frac{2w^2 - 12w + a_0(n)}{4y^3}\right),
\end{align*}
where
\begin{displaymath}
\widetilde{G_0}(x) = G_0(x) + \frac{x^2-3.85x + 14.15}{e^{3x}}- \frac{x^3 - 3.85x^2 + 14.15x}{e^{4x}} - \frac{2.85x}{e^{3x}}.
\end{displaymath}
Now we use \eqref{2.14} and \eqref{2.16} and collect all terms containing the numbers $2.85$ and $13.15$ to see that
\begin{align*}
\frac{2.5}{y^2} &- \frac{b_0(n)}{2y^2} + \left( 2.85 + \frac{13.15}{z} \right) \Phi_2(n) + \frac{70.7}{z^4} + \frac{458.7275}{z^5} + \frac{3428.7225}{z^6} 
+ \frac{w^2-w+1}{y^2z} \\
& \p{\q\q} + \frac{P_1(w)}{2y^3z} - \frac{P_2(w)}{6 y^4z} + \frac{P_3(w)}{12y^5z} - \frac{P_4(w)}{20y^6z} \\
& \geq \widetilde{G_0}(w) - \frac{a_0(n)}{2y^3} \left( 1 - \frac{w- 1}{y} - \frac{w-2}{y^2} + \frac{2w^2 - 12w + a_0(n)}{4y^3}\right).
\end{align*}
Applying \eqref{2.15} and Proposition \ref{prop204}, we get
\begin{align*}
\frac{2.5}{y^2} & - \frac{b_0(n)}{2y^2} + \frac{2.85}{z^2} + \frac{13.15}{z^3} + \frac{70.7}{z^4} + \frac{458.7275}{z^5} + \frac{3428.7225}{z^6} - \frac{1}{y} 
+ \frac{w}{y^2} + \frac{1}{z} \\
& \geq \widetilde{G_0}(w) - \frac{a_0(n)}{2y^3} \left( 1 - \frac{w- 1}{y} - \frac{w-2}{y^2} + \frac{2w^2 - 12w + a_0(n)}{4y^3}\right).
\end{align*}
A straightforward calculation shows that the last inequality is equivalent to
\begin{align*}
- \frac{1}{y} & - \frac{w^2 - 4w - (4-b_0(n))}{2y^2} + \frac{1}{z} + \frac{2.85}{z^2} + \frac{13.15}{z^3} + \frac{70.7}{z^4} + \frac{458.7275}{z^5} + 
\frac{3428.7225}{z^6} \nonumber\\
& \geq - \frac{w^2 - 6w+a_0(n)}{2y^3} - \frac{1}{2} \left( \frac{w - 1}{y} + \frac{w - 2}{y^2} - \frac{w^2 - 6w + a_0(n)}{2y^3} \right)^2 + \frac{1}{3} \left(
\frac{w - 1}{y} + \frac{w - 2}{y^2} \right)^3 \nonumber\\
& \p{\q\q} - \frac{1}{4} \left( \frac{w - 1}{y} \right)^4 + \frac{1}{5} \left(  \frac{w - 1}{y} \right)^5.
\end{align*}
We add $(w - 1)/y + (w - 2)/y^2$ to both sides of this inequality. Since $\log(1+x) \leq \sum_{k=1}^5 (-1)^{k+1}x/k$ for every $x > -1$, $g(x) = x^3/3$ is 
increasing, and $h(x)= - x^4/4 + x^5/5$ is decreasing on the interval $[0,1]$, we can use \eqref{2.7}--\eqref{2.9} to get
\begin{align*}
y & + w - 1 + \frac{w - 2}{y} - \frac{w^2 - 6w + b_0(n)}{2y^2} + \frac{1}{z} + \frac{2.85}{z^2} + \frac{13.15}{z^3} + \frac{70.7}{z^4} + \frac{458.7275}{z^5} + 
\frac{3428.7225}{z^6} \\
& \geq y + w - 1 + \log \left( 1 + \frac{w - 1}{y} + \frac{w - 2}{y^2} - \frac{w^2 - 6w + a_0(n)}{2y^3} \right).
\end{align*}
Finally, we use \eqref{2.10} and \eqref{2.12} to arrive at the desired result.
\end{proof}

Next we use Proposition \ref{prop207} and the following both lemmata to prove Theorem \ref{thm101}. In the irst lemma we determine the value of $N_0$ for $A_0 = 
0.87$.

\begin{lem} \label{lem202}
For every integer $n \geq 1\,338\,564\,587$, we have
\begin{displaymath}
\log n \geq 0.87 \log p_n.
\end{displaymath}
\end{lem}

\begin{proof}
We set
\begin{displaymath}
f(x) = e^x - 0.87 \left( e^x + x + \log \left( 1 + \frac{x-1}{e^x} + \frac{x-2}{e^{2x}} \right) \right).
\end{displaymath}
Since $f'(x) \geq 0$ for every $x \geq 2.5$ and $f(3.046) \geq 0.00137$, we see that $f(x) \geq 0$ for every $x \geq 3.046$. Substituting $x = \log \log n$ in 
$f(x)$ and using \eqref{1.10}, we see that the desired inequality holds for every $n \geq \exp(\exp(3.046))$. We check the remaining cases with a computer.
\end{proof}

Now we use Lemma \ref{lem202} to find the exact value of $N_1$ for $A_1 = 155.32$. 

\begin{lem} \label{lem203}
Let $A_1 = 155.32$. Then $F_1(n) \geq 0$ for every $n \geq 100\,720\,878$.
\end{lem}

\begin{proof}
First, let $n \geq \exp(\exp(3.05))$.
Since $f(x) = 6x^4 - 34.1x^3 + 163.65x^2 - 198.3x + 141.65 \geq 0$ for every $x \geq 0$,
it suffices to show that
\begin{equation}
\frac{155.32}{z^5} + \frac{6w^4 - 34.1w^3 + 268.2w^2 - 752.7w + 263.3}{6y^3z^2} + \frac{13.15w^2 - 83.85w + 13.15}{y^2 z^3} \geq 0. \tag{2.4} \label{2.4}
\end{equation}
In order to do this, we set
\begin{align*}
g(x) & = (6x^4 - 34.1x^3 + 268.2x^2 - 752.7x + 263.3)(e^x+x) \\
& \p{\q\q} + 6e^x(13.15x^2-83.85x + 13.15 + 155.32 \cdot 0.87^2).
\end{align*}
It is easy to see that $h_1(x) = 6x^4 - 10.1x^3 + 244.8x^2 - 561.6x - 208.229752 \geq 0$ for every $x \geq 2.6$ and that $h_2(x) = 30x^4 - 136.4x^3 + 804.6x^2 
- 1505.4x + 263.3 \geq 0$ for every $x \geq 2.2$. Hence $g'(x) = h_1(x)e^x + h_2(x) \geq 0$ for every $x \geq 2.6$. We also have $g(3.05) \geq 0.9$. Therefore, 
$g(x) \geq 0$ for every $x \geq 3.05$. Since $6x^4 - 34.1x^3 + 268.2x^2 - 752.7x + 263.3 \geq 0$ for every $x \geq 3.05$, we can use \eqref{1.3} to get 
$g(w)/(6y^3z^3) \geq 0$.
Now we apply Lemma \ref{lem202} to obtain \eqref{2.4}.
We finish by direct computation.
\end{proof}

Finally, we give a proof of Theorem \ref{thm101}.

\begin{proof}[Proof of Theorem \ref{thm101}]
For convenience, we write $w= \log \log n$ and $y=\log n$. Setting $A_0 = 0.87$ and $A_1 = 155.32$, we use Lemma \ref{lem202} and Lemma \ref{lem203} to get 
$N_0 =1\,338\,564\,587$ and $N_1 = 100\,720\,878$, respectively. The proof of this theorem goes in two steps.

\textit{Step 1}. We set $a_0(n) = - w^2+6w$. Then $N_2 = 688\,383$ is a suitable choice for $N_2$. By \eqref{2.11}, we get
\begin{equation}
b_0(n) \geq 10.7 + g(n), \tag{2.23} \label{2.23}
\end{equation}
where
\begin{align*}
g(n) & = - \frac{2w^3-18w^2+64.2w-98.9}{3y} + \frac{w^4 - 12w^3 + 63.16w^2 - 203.17w + 258.29}{2y^2} \\
& \p{\q\q} - \frac{2w^5 - 10w^4 + 30w^3 - 70w^2 + 90w - 1554.24}{5y^3} - \frac{8w^3 - 2137.44w^2 + 2185.45w - 37836.25}{12y^4}.
\end{align*}
We define
\begin{align*}
g_1(x,t) & = 3.54e^{4x} + 20(18x^2 + 98.9)e^{3x} -20(2t^3 + 64.2t)e^{3t} + 30(x^4 + 63.16x^2 + 258.29)e^{2x} \\
& \p{\q\q} - 30(12t^3 + 203.17t)e^{2t} + 12(10x^4 + 70x^2 + 1554.24)e^x - 12(2t^5 + 30t^3 + 90t)e^t \\
& \p{\q\q} + 5(2137.44x^2 + 37836.25) - 5(8t^3 + 2185.45t).
\end{align*}
If $t_0 \leq x \leq t_1$, then $g_1(x,x) \geq g_1(t_0,t_1)$. We check with a computer that $g_1(i \cdot 10^{-5}, (i+1) \cdot 10^{-5}) \geq 0$ for every 
integer $i$ with $0 \leq i \leq 699\,999$. Therefore,
\begin{equation}
g(n) + 0.059 = \frac{g_1(w,w)}{60y^4} \geq 0 \q\q (0 \leq w \leq 7). \tag{2.24} \label{2.24}
\end{equation}
Next we prove that $g_1(x,x) \geq 0$ for every $x \geq 7$. For this purpose, let $W_1(x) = 3.54e^x -20(2x^3 - 18x^2 + 64.2x - 98.9)$. It is easy to show that 
$W_1(x) \geq 792$ for every $x \geq 7$. Hence we get
\begin{align*}
g_1(x,x) & \geq (792e^x + 30(x^4 - 12x^3 + 63.16x^2 - 203.17x + 258.29))e^{2x} \\
& \p{\q\q} - 12(2x^5 - 10x^4 + 30x^3 - 70x^2 + 90x - 1554.24)e^x \\
& \p{\q\q} - 5(8x^3 - 2137.44x^2 + 2185.45x - 37836.25).
\end{align*}
Since $792e^t + 30(t^4 - 12t^3 + 63.16t^2 - 203.17t + 258.29) \geq 875\,011$ for every $t \geq 7$, we
obtain $g(n) + 0.059 = g_1(w,w)/(60y^4) \geq 0$ for $w \geq 7$. Combined with \eqref{2.24}, it gives that $g(n) \geq -0.059$ for every $n \geq 3$. 
Applying this to \eqref{2.23}, we get $b_0(n) \geq 10.641$ for every $n \geq 3$. Hence, by Proposition \ref{prop207}, we get
\begin{displaymath}
p_n < n \left( y + w - 1 + \frac{w - 2}{y} - \frac{w^2-6w + 10.641}{2y^2} \right)
\end{displaymath}
for every $n \geq 1\,338\,564\,587$. For every integer $n$ such that $39\,529\,802 \leq n \leq 1\,338\,564\,586$ we check the last inequality 
with a computer.

\textit{Step 2}. We set $a_0(n) = 10.641$. Using the first step, we can choose $N_2 = 39\,529\,802$. By \eqref{2.11}, we have
\begin{equation}
b_0(n) \geq 10.7 + h(n), \tag{2.25} \label{2.25}
\end{equation}
where $h(n)$ is given by
\begin{align*}
h(n) & = - \frac{2w^3-21w^2+82.2w-130.823}{3y} + \frac{w^4 - 14w^3 + 77.16w^2 - 236.45w + 279.57}{2y^2} \\
& \p{\q\q} - \frac{2w^5-10w^4+35w^3-110w^2+203.205w-1660.65}{5y^3} \\
& \p{\q\q} + \frac{3w^4 - 44w^3 + 2309.28w^2 - 2568.52w + 38175.947}{12y^4}.
\end{align*}
We set
\begin{align*}
h_1(x,t) & = 1.98e^{4x} + 20(21x^2+130.823)e^{3x} - 20(2t^3+82.2t)e^{3t} + 30(x^4 + 77.16x^2 + 279.57)e^{2x} \\
& \p{\q\q} - 30(14t^3 + 236.45t)e^{2t} + 12(10x^4+110x^2+1660.65)e^x - 12(2t^5+35t^3+203.205t)e^t \\
& \p{\q\q} + 5(3x^4 + 2309.28x^2 + 38175.947) - 5(44t^3 + 2568.52t).
\end{align*}
Clearly, $h_1(x,x) \geq h_1(t_0,t_1)$ for every $x$ such that $t_0 \leq x \leq t_1$. We use a computer to verify that $h_1(i \cdot 10^{-6}, (i+1) \cdot 
10^{-6}) 
\geq 0$ for every integer $i$ with $0 \leq i \leq 7\,999\,999$. Therefore,
\begin{equation}
h(n) + 0.033 = \frac{h_1(w,w)}{60y^4} \geq 0 \q\q (0 \leq w \leq 8). \tag{2.26} \label{2.26}
\end{equation}
We next show that $h_1(x,x) \geq 0$ for every $x \geq 8$. Since $1.98e^t -20(2t^3 - 21t^2 + 82.2t - 130.823) \geq 1766$ for every $t \geq 8$, we have
\begin{align*}
h_1(x,x) & \geq 1766e^{3x} + 30(x^4 - 14x^3 + 77.16x^2 - 236.45x + 279.57)e^{2x} \\
& \p{\q\q} - 12(2x^5 - 10x^4 + 35x^3 - 110x^2 + 203.205x - 1660.65)e^x \\
& \p{\q\q} + 5(3x^4 - 44x^3 + 2309.28x^2 - 2568.52x + 38175.947).
\end{align*}
Note that $1766e^t + 30(t^4 - 14t^3 + 77.16t^2 - 236.45t + 279.57) \geq 5\,271\,998$ for every $t \geq 8$. Hence $h(n) + 0.033 = h_1(w,w)/(60y^4) \geq 0$ for 
$w \geq 8$.
Combined with \eqref{2.26} and \eqref{2.25}, this gives $b_0(n) \geq 10.667$ for every $n \geq 3$. Applying this to Proposition \ref{prop501}, we complete 
the proof of the required inequality for every $n \geq 1\,338\,564\,587$. We verify the remaining cases with a computer.
\end{proof}

Under the assumption that the Riemann hypothesis is true, Dusart \cite[Theorem 3.4]{dusart2018} found that
\begin{equation}
p_n < n \left( \log n + \log \log n - 1 + \frac{\log \log n - 2}{\log n} - \frac{(\log \log n)^2-6\log \log n}{2 \log^2 n} \right). \tag{1.12} 
\label{1.12}
\end{equation}
for every integer $n \geq 3468$. Using Theorem \ref{thm101} and a computer for smaller values of $n$, we get

\begin{kor}
The inequality \eqref{1.12} holds unconditionally for every $n \geq 3468$.
\end{kor}

\section{Proof of Theorem \ref{thm102}}

We first introduce the polynomials $P_8, P_9, P_{10}, P_{11}, P_{12} \in \Q[x]$:
\begin{itemize}
\item $P_8(x) = 3x^2-6x+5.2$,
\item $P_9(x) = x^3-6x^2+11.4x-4.2$,
\item $P_{10}(x) = 2x^3-7.2x^2+8.4x-4.41$,
\item $P_{11}(x) = x^3-4.2x^2+4.41x$,
\item $P_{12}(x) = 9.3x^2-12.3x+11.5$.
\end{itemize}
Further, let the polynomials $Q_7, Q_8, Q_9$ with rational coefficients are given by
\begin{itemize}
\item $Q_7(x) = (x^2-x+1)P_{12}(x) + (x^2-x+1)P_8(x) - 3.15P_9(x) - P_{10}(x) + 12.85P_8(x)$,
\item $Q_8(x) = 3.15P_{10}(x) + 12.85P_9(x)$,
\item $Q_9(x) = 2(x^2-x+1)P_9(x) - P_8(x)P_{12}(x)$.
\end{itemize}
Let $B_1, \ldots, B_{10}$ be real positive constants satisfying
\begin{equation} 
B_6 + B_7 + B_8 + B_9 + B_{10} \leq 3.15. \tag{3.2} \label{3.2}
\end{equation}
Writing $w = \log \log n$, $y = \log n$, and $z = \log p_n$, we define $H_i : \N_{\geq 2} \rightarrow \R$, where $1 \leq i \leq 10$, by
\begin{itemize}
\item $\displaystyle H_1(n) = \frac{B_1w}{y^3z} - \frac{Q_7(w)}{2y^5z} + \frac{Q_8(w)}{2y^5z^2} + \frac{Q_9(w)}{4y^6z} + \frac{12.85P_9(w)}{2y^4z^3}$,
\item $\displaystyle H_2(n) = \frac{B_2w}{y^3z} + \frac{12.85w}{y^2z^2} - \frac{71.3}{z^4}$,
\item $\displaystyle H_3(n) = \frac{B_3w}{y^3z} - \frac{3.15P_8(w)}{2y^3z^2} - \frac{12.85(w^2 - w+1)}{y^3z^2}$,
\item $\displaystyle H_4(n) = \frac{B_4w}{y^3z} + \frac{3.15P_9(w) - 12.85P_8(w)}{2y^4z^2}$,
\item $\displaystyle H_5(n) = \frac{B_5w}{y^3z} + \frac{P_9(w) - 3.15P_8(w)}{2y^4z} - \frac{12.85(w^2 - w +1)}{y^4z} - \frac{(w^2-w+1)^2}{y^4z}$,
\item $\displaystyle H_6(n) = \frac{B_6w}{y^2z} + \frac{(12.85-B_1-B_2-B_3-B_4-B_5)w}{y^3z} - \frac{3.15(w^2 - w +1)}{y^2z^2}$,
\item $\displaystyle H_7(n) = \frac{B_7w}{y^2z} - \frac{12.85 P_8(w)}{2y^3z^3}$,
\item $\displaystyle H_8(n) = \frac{B_8w}{y^2z} - \frac{12.85(w^2 - w +1)}{y^2z^3}$,
\item $\displaystyle H_9(n) = \frac{B_9w}{y^2z} - \frac{463.2275}{z^5}$,
\item $\displaystyle H_{10}(n) = \frac{B_{10}w}{y^2z} - \frac{4585}{z^6}$.
\end{itemize}
Since $H_i(n)$, $1 \leq i \leq 10$, is nonnegative for all sufficiently large values of $n$, we can define
\begin{displaymath}
M_i = M_i(B_i) = \min \{ k \in \N \mid H_i(n) \geq 0 \text{ for every } n \geq k \}
\end{displaymath}
and set $K_1 = \max_{1 \leq i \leq 10} M_i$. Let $a_1 : \N_{\geq 2} \rightarrow \R$ be an arithmetic function and let $K_2$ be a positive integer, which 
depends on $a_1$, so that the inequalities
\begin{equation}
a_1(n) > - (\log \log n)^2 + 6\log \log n, \tag{3.3} \label{3.3}
\end{equation}
\begin{equation}
0 \leq \frac{\log \log n-1}{\log n} + \frac{\log \log n - 2}{\log^2 n} - \frac{(\log \log n)^2 - 6\log \log n+a_1(n)}{2\log^3 n} \leq 1, \q \text{and}  
\tag{3.4} \label{3.4}
\end{equation}
\begin{equation}
p_n > n \left( \log n + \log \log n - 1 + \frac{\log \log n - 2}{\log n} - \frac{(\log \log n)^2-6\log \log n+a_1(n)}{2\log^2 n} \right) \tag{3.5} \label{3.5}
\end{equation}
hold simultaneously for every $n \geq K_2$. Furthermore, we define the function $G_1 : \R \rightarrow \R$ by
\begin{align*}
G_1(x) & = \frac{2x^3 - 15x^2 + 42x - 14}{6e^{3x}} + \frac{3.15x}{e^{3x}} - \frac{12.85}{e^{3x}} - \frac{x^2 - x + 1}{e^{3x}} + \frac{(x^2-x+1)x}{e^{4x}} - 
\frac{P_{12}(x)}{2e^{4x}} + \frac{12.85x}{e^{4x}} \\
& \p{\q\q} + \frac{P_{12}(x)x}{2e^{5x}} + \frac{(x-1)^2}{2e^{2x}} - \frac{x^3-6x^2+12x-7}{3e^{3x}} - \sum_{k=2}^4 \frac{(-1)^k}{k} \left( \frac{x-1}{e^x} + 
\frac{x-2}{e^{2x}} \right)^k  + \frac{(x-2)^4}{4e^{8x}}.
\end{align*}
In order to prove Theorem \ref{thm102}, we set
\begin{equation}
b_1(n) = 11.3 - 2\,G_1(\log \log n) \log^2 n + \frac{a_1(n)}{\log n} - \frac{2A_0(3.15-(B_6 + B_7 + B_8 + B_9 + B_{10}))\log \log n}{\log n} \tag{3.6} 
\label{3.6}
\end{equation}
and prove the following proposition.

\begin{prop} \label{prop306}
For every integer $n \geq \max \{N_0, K_1, K_2, 3520\}$,  we have
\begin{displaymath}
p_n > n \left( \log n + \log \log n - 1 + \frac{\log \log n - 2}{\log n} - \frac{(\log \log n)^2-6\log \log n + b_1(n)}{2 \log^2 n} \right).
\end{displaymath}
\end{prop}

The following two lemmata are helpful for the proof of Proposition \ref{prop306}. The proofs are left to the reader.

\begin{lem} \label{lem301}
For every integer $n \geq 6$, we have
\begin{displaymath}
\frac{12.85P_9(\log \log n)}{2\log^6 n \log p_n} + \frac{3.15P_{10}(\log \log n)}{2\log^6 n \log p_n} + \frac{P_{11}(\log \log n)}{2\log^6 n \log p_n} \geq 0.
\end{displaymath}
\end{lem}


\begin{lem} \label{lem302}
Let $w = \log \log n$. For every integer $n \geq 17$, we have
\begin{displaymath}
\frac{P_9(w)P_{12}(w)}{4\log^7n\log p_n} + \frac{12.85P_{10}(w)}{2\log^7n\log p_n} + \frac{3.15P_{11}(w)}{2\log^7n\log p_n} + 
\frac{3.15P_{11}(w)}{2\log^6n\log^2p_n} \geq \frac{(w - 2)^4}{4\log^8n}.
\end{displaymath}
\end{lem}


Now we give a proof of Proposition \ref{prop306}.

\begin{proof}[Proof of Proposition \ref{prop306}]
Let $n \geq \max \{N_0, K_1, K_2, 3520\}$. By \cite[Theorem 3.2]{ca2017}, we have 
\begin{equation}
p_n > n \left( \log p_n - 1 - \frac{1}{\log p_n} - \frac{3.15}{\log^2p_n} - \frac{12.85}{\log^3p_n} - \frac{71.3}{\log^4p_n} - \frac{463.2275}{\log^5p_n} - 
\frac{4585}{\log^6p_n} \right). \tag{3.7} \label{3.7}
\end{equation}
For convenience, we write $w = \log \log n$, $y = \log n$, and $z = \log p_n$. By Corollary \ref{kor305}, we have
\begin{equation}
- \frac{1}{z} \geq \Psi_1(n) = - \frac{1}{y} + \frac{w}{y^2} - \frac{w^2-w+1}{y^2z} - \frac{P_8(w)}{2y^3z} + \frac{P_9(w)}{2y^4z}. \tag{3.8} \label{3.8}
\end{equation}
Similarly to the proof of \eqref{2.15}, we use Proposition \ref{prop303} to get
\begin{equation}
- \frac{1}{z^2} \geq \Psi_2(n) = - \frac{1}{y^2} + \frac{w}{y^3} + \frac{w}{y^2z} - \left( \frac{1}{y} + \frac{1}{z} \right) \left(\frac{w^2 - w + 1}{y^2z} + 
\frac{P_8(w)}{2y^3z} - \frac{1}{2z} \sum_{k=4}^6\frac{P_{k+5}(w)}{y^k} \right). \tag{3.9} \label{3.9}
\end{equation}
Using $P_{11}(x) = x(x-2.1)^2 \geq 0$ for every $x \geq 0$, $P_{10}(x) = 2(x-2.1)(x^2-1.5x+1.05) \geq 0$ for every $x \geq 2.1$ and Corollary 
\ref{kor304}, we get
\begin{align}
- \frac{1}{z^3} \geq \Psi_3(n) & = - \frac{1}{y^3} + \frac{w}{y^4} + \frac{w}{y^3z} + \frac{w}{y^2z^2} - \frac{w^2 - w + 1}{y^4z} - \frac{w^2 - w + 1}{y^3z^2} - 
\frac{w^2 - w + 1}{y^2z^3} \tag{3.10} \label{3.10} \\
& \p{\q\q} - \frac{P_8(w)}{2y^5z} - \frac{P_8(w)}{2y^4z^2} - \frac{P_8(w)}{2y^3z^3} + \frac{P_9(w)}{2y^6z} + \frac{P_9(w)}{2y^5z^2} + \frac{P_9(w)}{2y^4z^3} + 
\frac{P_{10}(w)}{2y^7z}. \nonumber
\end{align}
By \eqref{3.2}, $3.15 - (B_6 + B_7 + B_8 + B_9 + B_{10}) \geq 0$. Since $n \geq N_0$ is assumed, we have $F_0(n) \geq 0$. Hence, by \eqref{2.17} and 
\eqref{3.6}, we see that
\begin{equation}
\frac{11.3 - b_1(n)}{2y^2} \leq G_1(w) - \frac{a_1(n)}{2y^3} + \frac{(3.15 - (B_6 + B_7 + B_8 + B_9 + B_{10}))w}{y^2z}. \tag{3.11} \label{3.11}
\end{equation}
We have $n \geq K_1$. This means that $\sum_{i=1}^{10} H_i(n) \geq 0$. So we can add $\sum_{i=1}^{10} H_i(n)$ to the right-hand side of \eqref{3.11} and use 
Lemmata \ref{lem301} and \ref{lem302} to get
\begin{align*}
\frac{d(n)}{2y^2} & \leq G_1(w) - \frac{a_1(n)}{2y^3} + \frac{3.15w}{y^2z} + \frac{12.85P_9(w)}{2y^6z} + \frac{3.15P_{10}(w)}{2y^6z} + \frac{P_{11}(w)}{2y^6z} 
- \frac{Q_7(w)}{2y^5z} + \frac{Q_8(w)}{2y^5z^2} + \frac{Q_9(w)}{4y^6z} \\
& \p{\q\q} +  \frac{12.85P_9(w)}{2y^4z^3} + \frac{12.85w}{y^2z^2} - \frac{71.3}{z^4} - \frac{3.15P_8(w)}{2y^3z^2} - \frac{12.85(w^2-w+1)}{y^3z^2} + 
\frac{3.15P_9(w)}{2y^4z^2} \\
& \p{\q\q} - \frac{12.85P_8(w)}{2y^4z^2} + \frac{P_9(w)}{2y^4z} - \frac{3.15P_8(w)}{2y^4z} - \frac{12.85(w^2-w+1)}{y^4z} - \frac{(w^2-w+1)^2}{y^4z} + 
\frac{12.85w}{y^3z} \\
& \p{\q\q} - \frac{3.15(w^2 - w +1)}{y^2z^2} - \frac{12.85 P_8(w)}{2y^3z^3} - \frac{12.85(w^2-w+1)}{y^2z^3} - \frac{463.2275}{z^5} - \frac{4585}{z^6} \\
& \p{\q\q} + \frac{P_9(w)P_{12}(w)}{4y^7z} + \frac{12.85P_{10}(w)}{2y^7z} + \frac{3.15P_{11}(w)}{2y^7z}+ \frac{3.15P_{11}(w)}{2y^6z^2} - \frac{(w - 2)^4}{4y^8},
\end{align*}
where $d(n) = 11.3 - b_1(n)$. Applying the defining formulas of $Q_7, Q_8, Q_9$, and $G_1$ to the last inequality, we find
\begin{align*}
\frac{d(n)}{2y^2} & \leq - \frac{a_1(n)}{2y^3} + \frac{2w^3 - 15w^2 + 42w - 14}{6y^3} + \frac{3.15w}{y^3} + 12.85 \cdot \Psi_3(n) + \left( \frac{w^2-w+1}{y^2} 
+ \frac{P_{12}(w)}{2y^3} \right) \cdot \Psi_1(n) \\
& \p{\q\q} + \frac{(w-1)^2}{2y^2} - \frac{w^3-6w^2+12w-7}{3y^3} - \sum_{k=2}^4 \frac{(-1)^k}{k} \left( \frac{w-1}{y} + \frac{w-2}{y^2} \right)^k + 
\frac{3.15w}{y^2z} + \frac{3.15P_{10}(w)}{2y^6z} \\
& \p{\q\q} + \frac{P_{11}(w)}{2y^6z} + \frac{3.15P_9(w)}{2y^5z} + \frac{P_{10}(w)}{2y^5z} + \frac{3.15P_{10}(w)}{2y^5z^2} - 
\frac{71.3}{z^4} - \frac{3.15P_8(w)}{2y^3z^2} + \frac{3.15P_9(w)}{2y^4z^2} + \frac{P_9(w)}{2y^4z} \\
& \p{\q\q} - \frac{3.15P_8(w)}{2y^4z} - \frac{3.15(w^2-w+1)}{y^2z^2} - \frac{463.2275}{z^5} - \frac{4585}{z^6} + \frac{3.15P_{11}(w)}{2y^7z}+ 
\frac{3.15P_{11}(w)}{2y^6z^2},
\end{align*}
where $\Psi_1(n)$ and $\Psi_3(n)$ are given as in beginning of the proof. Note that $w^2-w+1$ and $P_{12}(w)$ are nonnegative. Therefore, we can apply 
\eqref{3.8} and \eqref{3.10} to the last inequality and get
\begin{align*}
\frac{d(n)}{2y^2} & \leq \frac{2w^3 - 15w^2 + 42w - 14}{6y^3} + \frac{3.15w}{y^3} - \frac{12.85}{z^3} - \frac{w^2-w+1}{y^2z} - \frac{P_{12}(w)}{2y^3z} + 
\frac{(w-1)^2}{2y^2} \\
& \p{\q\q} - \frac{w^3-6w^2+12w-7}{3y^3} - \sum_{k=2}^4 \frac{(-1)^k}{k} \left( \frac{w-1}{y} + \frac{w-2}{y^2} \right)^k - \frac{a_1(n)}{2y^3} + 
\frac{3.15w}{y^2z} + \frac{3.15P_{10}(w)}{2y^6z} \\
& \p{\q\q} + \frac{P_{11}(w)}{2y^6z} + \frac{3.15P_9(w)}{2y^5z} + \frac{P_{10}(w)}{2y^5z} + \frac{3.15P_{10}(w)}{2y^5z^2} - \frac{71.3}{z^4} - 
\frac{3.15P_8(w)}{2y^3z^2} + \frac{3.15P_9(w)}{2y^4z^2} + \frac{P_9(w)}{2y^4z} \\
& \p{\q\q} - \frac{3.15P_8(w)}{2y^4z} - \frac{3.15(w^2-w+1)}{y^2z^2} - \frac{463.2275}{z^5} - \frac{4585}{z^6} + \frac{3.15P_{11}(w)}{2y^7z} + 
\frac{3.15P_{11}(w)}{2y^6z^2}.
\end{align*}
Since $P_{12}(x) = P_8(x) + 2 \cdot 3.15(x^2-x+1)$ and $d(n) = 11.3 - b_1(n)$, the last inequality is equivalent to
\begin{align*}
\frac{5-b_1(n)}{2y^2} & \leq 3.15 \cdot \Psi_2(n) + \frac{2w^3 - 15w^2 + 42w - 14}{6y^3} - \frac{12.85}{z^3} - \frac{w^2-w+1}{y^2z} - \frac{P_8(w)}{2y^3z} + 
\frac{(w-1)^2}{2y^2} \\
& \p{\q\q} - \frac{w^3-6w^2+12w-7}{3y^3} - \sum_{k=2}^4 \frac{(-1)^k}{k} \left( \frac{w-1}{y} + \frac{w-2}{y^2} \right)^k - \frac{a_1(n)}{2y^3} + 
\frac{P_9(w)}{2y^4z} + \frac{P_{10}(w)}{2y^5z}  \\
& \p{\q\q} + \frac{P_{11}(w)}{2y^6z} - \frac{71.3}{z^4} - \frac{463.2275}{z^5} - \frac{4585}{z^6}.
\end{align*}
Using \eqref{3.9} and Proposition \ref{prop303}, we get the inequality
\begin{align*}
\frac{5-b_1(n)}{2y^2} & \leq - \frac{1}{z} - \frac{3.15}{z^2} - \frac{12.85}{z^3} - \frac{71.3}{z^4} - \frac{463.2275}{z^5} -  \frac{4585}{z^6} + \frac{2w^3 - 
15w^2+42w-14}{6y^3} + \frac{1}{y} - \frac{w}{y^2} \\
& \p{\q\q} + \frac{(w-1)^2}{2y^2} - \frac{w^3-6w^2+12w-7}{3y^3} - \sum_{k=2}^4 \frac{(-1)^k}{k} \left( \frac{w-1}{y} + \frac{w-2}{y^2} \right)^k  - 
\frac{a_1(n)}{2y^3}
\end{align*}
which is equivalent to
\begin{align}
\frac{w-2}{y} & \leq \frac{w-1}{y} + \frac{w-2}{y^2} - \frac{w^2-6w + a_1(n)}{2y^3} - \sum_{k=2}^4 \frac{(-1)^k}{k} \left( \frac{w-1}{y} + \frac{w-2}{y^2} 
\right)^k \tag{3.12} \label{3.12} \\
& \p{\q\q} + \frac{w^2-6w+b_1(n)}{2y^2} - \frac{1}{z} - \frac{3.15}{z^2} - \frac{12.85}{z^3} - \frac{71.3}{z^4} - \frac{463.2275}{z^5} - \frac{4585}{z^6}. 
\nonumber
\end{align}
Since $\log(1+t) \geq \sum_{k=1}^4 (-1)^{k+1}t^k/k$ for every $t > -1$ and both $g_1(x) = -x^2/2 + x^3/3$ and $g_2(x) = -x^4/4$ are decreasing on the interval 
$[0,1]$, we can use \eqref{3.3} and \eqref{3.4} to see that the inequality \eqref{3.12} implies
\begin{align*}
\frac{w-2}{y} - \frac{w^2 - 6w + b_1(n)}{2y^2} & \leq \log \left(1 + \frac{w-1}{y} + \frac{w - 2}{y^2} - \frac{w^2-6w + a_1(n)}{2y^3} \right) - \frac{1}{z} - 
\frac{3.15}{z^2} - \frac{12.85}{z^3} \\
& \p{\q\q} - \frac{71.3}{z^4}- \frac{463.2275}{z^5} - \frac{4585}{z^6}.
\end{align*}
Now we add $y+w-1$ to both sides of the last inequality und use \eqref{3.5} to get
\begin{displaymath}
y + w - 1 + \frac{w - 2}{\log n} - \frac{w^2 - 6w + b_1(n)}{y} \leq z -1 - \frac{1}{z} - \frac{3.15}{z^2} - \frac{12.85}{z^3} - \frac{71.3}{z^4} - 
\frac{463.2275}{z^5} - \frac{4585}{z^6}.
\end{displaymath}
Finally, we multiply the last inequality by $n$ and apply \eqref{3.7} to complete the proof.
\end{proof}

Now, we give a proof of Theorem \ref{thm102}.


\begin{proof}[Proof of Theorem \ref{thm102}]
Let $A_0 = 0.87$. Then, by Lemma \ref{lem202}, $N_0 = 1\,338\,564\,587$. In the following table we give the explicit values of $M_i$ for given $B_i$:
\vspace{2mm}
\begin{center}
\begin{tabular}{|c||c|c|c|c|c|c|}
\hline
  $i$ &                1&                2&                3&                4&                5 \\ \hline
$B_i$ &             0.27&             4.23&            1.575&            0.058&             2.24 \\ \hline
$M_i$ & 1\,359\,056\,314& 1\,471\,247\,583& 1\,468\,111\,666& 1\,383\,728\,153& 1\,462\,324\,835 \\ \hline \hline
  $i$ &                6&                7&                8&                9&               10 \\ \hline
$B_i$ &            0.105&           0.0026&            0.052&           0.1955&             0.08 \\ \hline
$M_i$ &                5& 1\,075\,859\,481& 1\,445\,815\,789& 1\,479\,240\,488& 1\,447\,605\,594 \\ \hline
\end{tabular}
\end{center}
\vspace{2mm}
The proof that $H_i(n) \geq 0$ for every $n \geq M_i$ and each integer $i$ satisfying $1 \leq i \leq 10$ can be found in Section 6. The above table indicates
\begin{equation}
3.15 - (B_6 + B_7 + B_8 + B_9 + B_{10}) = 2.7149 \tag{3.13} \label{3.13}
\end{equation}
and $K_1 = \max_{1 \leq i \leq 10} M_i = 1\,479\,240\,488$.

\textit{Step 1}. We set $a_1(n) = 0.2 y - w^2 + 6w$. Then, by \eqref{1.11} and \eqref{3.3}--\eqref{3.5}, we can choose $K_2 = 33$. Using \eqref{3.6} and 
\eqref{3.13}, we obtain
\begin{align*}
b_1(n) & = 11.5 - \frac{2w^3-18w^2+63.071778w-97.1}{3y} + \frac{w^4-12w^3+46.6w^2-112w+40}{2y^2} \\
& \p{\q\q} + \frac{2w^4-21.3w^3+40.3w^2-41.5w+12}{y^3} + \frac{9w^4-56w^3+129w^2-132w+52}{3y^4} \\
& \p{\q\q} + \frac{2w^4-14w^3+36w^2-40w+16}{y^5}.
\end{align*}
In this step, we show that $b_1(n) \leq 11.589$ for every $n \geq \exp(\exp(3.05))$. For this purpose, we set
\begin{align*}
\alpha(x,t) & = 0.534e^{5x} + 2(2x^3 + 63.071778x)e^{4x} - 2(18t^2+97.1)e^{4t} + 3(12x^3 + 112x)e^{3x} \\
& \p{\q\q} - 3(t^4 + 46.6t^2 + 40)e^{3t} + 6(21.3x^3 + 41.5x)e^{2x} - 6(2t^4 + 40.3t^2 + 12)e^{2t} \\
& \p{\q\q} + 2(56x^3 + 132x)e^x  - 2(9t^4 + 129t^2 + 52)e^t + 6(14x^3 + 40x) - 6(2t^4 + 36t^2 + 16).
\end{align*}
Note that this function satisfies the identity 
\begin{equation}
\alpha(w,w) = 6(11.589 - b_1(n))y^5. \tag{3.14} \label{3.14}
\end{equation}
If $t_0 \leq x \leq t_1$, then $\alpha(x,x) \geq \alpha(t_0,t_1)$. We check with a computer that $\alpha(3.05 + i \cdot 10^{-5}, 3.05 + (i+1) \cdot 10^{-5}) 
\geq 0$ for every integer $i$ satisfying $0 \leq i \leq 394\,999$. Hence by \eqref{3.14}, 
\begin{equation}
b_1(n) \leq 11.589 \q\q (3.05 \leq w \leq 7). \tag{3.15} \label{3.15}
\end{equation}
Next, we show that $\alpha(x,x) \geq 0$ for every $x \geq 7$. Since $0.534e^x + 2(2x^3-18x^2+63.071778x-97.1) \geq 882$ for every $x \geq 7$, we have
\begin{align*}
\alpha(x,x) & \geq 882e^{4x} - 3(x^4-12x^3+46.6x^2-112x+40)e^{3x} - 6(2x^4-21.3x^3+40.3x^2-41.5x+12)e^{2x} \\
& \p{\q\q} - 2(9x^4-56x^3+129x^2-132x+52)e^x - 6(2x^4-14x^3+36x^2-40x+16).
\end{align*}
Note that $882e^{x} - 3(x^4-12x^3+46.6x^2-112x+40) \geq 967\,757$ for every $x \geq 7$. Therefore, $\alpha(x,x) \geq 0$ for every $x \geq 7$. Combined with 
\eqref{3.14} and \eqref{3.15}, it gives $b_1(n) \leq 11.589$ for every $n \geq \exp(\exp(3.05))$. Applying this to Proposition \ref{prop306}, we get
\begin{displaymath}
p_n > n \left( y + w - 1 + \frac{w - 2}{y} - \frac{w^2-6w + 11.589}{2y^2} \right)
\end{displaymath}
for every $n \geq \exp(\exp(3.05))$. We check with a computer that the last inequality also holds for every integer $n$ such that $2 \leq n \leq 
\exp(\exp(3.05))$.

\textit{Step 2}. We set $a_1(n) = 11.589$. Then $K_2 = 48$ is a suitable choice for $K_2$. Combined with \eqref{3.6} and \eqref{3.13}, it gives
\begin{align*}
b_1(n) & = 11.3 - \frac{2w^3 - 21w^2 + 81.071778w - 131.867}{3y} + \frac{w^4 - 12w^3 + 46.6w^2 - 112w + 40}{2y^2} \\
& \p{\q\q} + \frac{2w^4 - 21.3w^3 + 40.3w^2 - 41.5w + 12}{y^3} + \frac{9w^4 - 56w^3 + 129w^2 - 132w + 52}{3y^4} \\
& \p{\q\q} + \frac{2w^4 - 14w^3 + 36w^2 - 40w + 16}{y^5}.
\end{align*}
We set
\begin{align*}
\beta(x,t) & = 1.272e^{5x} + 2(2x^3 + 81.071778x)e^{4x} - 2(21t^2 + 131.867)e^{4t} + 3(12x^3 + 112x)e^{3x} \\
& \p{\q\q} - 3(t^4 + 46.6t^2 + 40)e^{3t} + 6(21.3x^3 + 41.5x)e^{2x} - 6(2t^4 + 40.3t^2 + 12)e^{2t} \\
& \p{\q\q} + 2(56x^3 + 132x)e^x  - 2(9t^4 + 129t^2 + 52)e^t + 6(14x^3 + 40x) - 6(2t^4 + 36t^2 + 16).
\end{align*}
Then $\beta(w,w) = 6(11.512 - b_1(n))y^5$. Similarly to the first step, we get
\begin{displaymath}
b_1(n) \leq 11.512 \q\q (3.05 \leq w \leq 7).
\end{displaymath}
So it suffices to verify that $\beta(x,x) \geq 0$ for every $x \geq 7$. Notice that $1.272e^x + 2(2x^3-21x^2+81.071778x-131.867) \geq 1580$ for every $x 
\geq 7$. Thus we get
\begin{align*}
\beta(x,x) & \geq 1580e^{4x} - 3(x^4-12x^3+46.6x^2-112x+40)e^{3x} - 6(2x^4-21.3x^3+40.3x^2-41.5x+12)e^{2x} \\
& \p{\q\q} - 2(9x^4-56x^3+129x^2-132x+52)e^x - 6(2x^4-14x^3+36x^2-40x+16).
\end{align*}
Since $1580e^x - 3(x^4-12x^3+46.6x^2-112x+40) \geq 1\,733\,207$ for every $x \geq 7$, we conclude that $\beta(x,x) \geq 0$ for every $x \geq 7$. Hence $b_1(n) 
\leq 11.512$ for every $n \geq \exp(\exp(3.05))$. Therefore, by Proposition \ref{prop306},
\begin{displaymath}
p_n > n \left( y + w - 1 + \frac{w - 2}{y} - \frac{w^2-6w + 11.512}{2y^2} \right)
\end{displaymath}
for every $n \geq \exp(\exp(3.05))$. Finally, we use a computer to verify that the last inequality also holds for every integer $n$ such that $2 \leq n \leq 
\exp(\exp(3.05))$.

\textit{Step 3}. Here we set $a_1(n) = 11.512$. Then we can choose $K_2 = 47$. By \eqref{3.6} and \eqref{3.13},
\begin{align*}
b_1(n) & = 11.3 - \frac{2w^3 - 21w^2 + 81.071778w - 131.636}{3y} + \frac{w^4 - 12w^3 + 46.6w^2 - 112w + 40}{2y^2} \\
& \p{\q\q} + \frac{2w^4 - 21.3w^3 + 40.3w^2 - 41.5w + 12}{y^3} + \frac{9w^4 - 56w^3 + 129w^2 - 132w + 52}{3y^4} \\
& \p{\q\q} + \frac{2w^4 - 14w^3 + 36w^2 - 40w + 16}{y^5}.
\end{align*}
To show that $b_1(n) \leq 11.508$ for every $n \geq \exp(\exp(3.05))$, we set
\begin{align*}
\gamma(x,t) & = 1.248e^{5x} + 2(2x^3 + 81.071778x)e^{4x} - 2(21t^2 + 131.636)e^{4t} + 3(12x^3 + 112x)e^{3x} \\
& \p{\q\q} - 3(t^4 + 46.6t^2 + 40)e^{3t} + 6(21.3x^3 + 41.5x)e^{2x} - 6(2t^4 + 40.3t^2 + 12)e^{2t} \\
& \p{\q\q} + 2(56x^3 + 132x)e^x  - 2(9t^4 + 129t^2 + 52)e^t + 6(14x^3 + 40x) - 6(2t^4 + 36t^2 + 16).
\end{align*}
Notice that $\gamma(w,w) = 6(11.508 - b_1(n))y^5$.
Analogously to the first step, we obtain $b_1(n) \leq 11.508$ for $w$ satisfying $3.05 \leq w \leq 7$.
Next we find $b_1(n) \leq 11.508$ for $w \geq 7$. Note that $1.248e^x + 2(2x^3-21x^2+81.071778x-131.636) \geq 1554$ for every $x 
\geq 7$. Therefore,
\begin{align*}
\gamma(x,x) & \geq 1554e^{4x} - 3(x^4-12x^3+46.6x^2-112x+40)e^{3x} - 6(2x^4-21.3x^3+40.3x^2-41.5x+12)e^{2x} \\
& \p{\q\q} - 2(9x^4-56x^3+129x^2-132x+52)e^x - 6(2x^4-14x^3+36x^2-40x+16).
\end{align*}
Since $1554e^{x} - 3(x^4-12x^3+46.6x^2-112x+40) \geq 1\,704\,694$ for every $x \geq 7$, we get $\gamma(x,x) \geq 0$ for every $x \geq 7$. Hence $b_1(n) \leq 
11.508$ for every $n \geq \exp(\exp(3.05))$. Then Proposition \ref{prop306} implies the required inequality for every $n \geq \exp(\exp(3.05))$. We complete 
the proof by direct computation.
\end{proof}

We get the following corollary which was already known under the assumption that the Riemann hypothesis is true (see Dusart \cite[Theorem 3.4]{dusart2018}).

\begin{kor}
For every $n \geq 2$, we have
\begin{displaymath}
p_n > n \left( \log n + \log \log n - 1 + \frac{\log \log n - 2}{\log n} - \frac{(\log\log n)^2}{2\log^2 n} \right).
\end{displaymath}
\end{kor}

\begin{proof}
For $n \geq 905$, this corollary is a direct consequence of Theorem \ref{thm102}. For smaller values of $n$, we check the requiered inequality with a computer.
\end{proof}

\begin{rema}
Compared with Theorem \ref{thm102}, the asymptotic expansion \eqref{1.2} implies a better lower bound for the $n$th prime number, which corresponds to 
the first five terms, namely that
\begin{equation}
p_n > n \left( \log n + \log \log n - 1 + \frac{\log \log n - 2}{\log n} - \frac{(\log\log n)^2-6\log \log n + 11}{2\log^2 n} \right) \tag{3.16} \label{3.16}
\end{equation}
for all sufficiently large values of $n$. Let $r_3$ denote the smallest positive integer such that the inequality \eqref{3.16} holds for every $n \geq r_3$. 
Under the assumption that the Riemann hypothesis is true, Arias de Reyna and Toulisse \cite[Theorem 6.4]{adrt} proved that $3.9 \cdot 10^{30} < r_3 \leq 3.958 
\cdot 10^{30}$.
\end{rema}

\section{Estimates for $\vartheta(p_n)$}

Chebyshev's $\vartheta$-function is defined by
\begin{displaymath}
\vartheta(x) = \sum_{p \leq x} \log p,
\end{displaymath}
where $p$ runs over primes not exceeding $x$. Notice that the Prime Number Theorem is equivalent to
\begin{equation}
\vartheta(x) \sim x \quad\quad (x \to \infty). \tag{4.1} \label{4.1}
\end{equation}
By proving the existence of a zero-free region for the Riemann zeta-function $\zeta(s)$ to the left of the line $\text{Re}(s) = 1$, de la Vall\'{e}e-Poussin
\cite{vallee1899} found an estimate for the error term in \eqref{4.1} by proving $\vartheta(x) = x + O(x \exp(-c\sqrt{\log x}))$, where $c$ is a positive 
absolute constant. Applying \eqref{1.2} to the last asymptotic formula, we see that
\begin{displaymath}
\vartheta(p_n) = n \left( \log n + \log \log n - 1 + \frac{\log \log n - 2}{\log n} - \frac{(\log \log n)^2 - 6 \log \log n + 11}{2 \log^2 n} + O\left( 
\frac{(\log \log n)^3}{\log^3n} \right)\right).
\end{displaymath}
In this direction, many estimates for $\vartheta(p_n)$ were obtained (see for example Massias and Robin \cite[Th\'{e}or\`{e}me B]{mr}). The current best ones
are due to Dusart \cite[Propositions 5.11 and 5.12]{dusart2017}. He found that
\begin{displaymath}
\vartheta(p_n) \geq n \left( \log n + \log \log n - 1 + \frac{\log \log n - 2.04}{\log n} \right)
\end{displaymath}
for every $n \geq \pi(10^{15}) + 1 = 29\,844\,570\,422\,670$, and that the inequality
\begin{displaymath}
\vartheta(p_n) \leq n \left( \log n + \log \log n - 1 + \frac{\log \log n - 2}{\log n} - \frac{0.782}{\log^2 n} \right)
\end{displaymath}
holds for every $n \geq 781$. Using Theorems \ref{thm101} and \ref{thm102}, we find the following estimates for $\vartheta(p_n)$, which improve the estimates 
given by Dusart.

\begin{prop} \label{prop801}
For every integer $n \geq 2$, we have
\begin{displaymath}
\vartheta(p_n) > n \left( \log n + \log \log n - 1 + \frac{\log \log n - 2}{\log n} - \frac{(\log \log n)^2 - 6 \log \log n + 11.808}{2 \log^2 n} \right),
\end{displaymath}
and for every integer $n \geq 2581$, we have
\begin{displaymath}
\vartheta(p_n) < n \left( \log n + \log \log n - 1 + \frac{\log \log n - 2}{\log n} - \frac{(\log \log n)^2 - 6 \log \log n + 10.367}{2 \log^2 n} \right).
\end{displaymath}
\end{prop}

\begin{proof}
First, we consider the case where $n \geq 841\,508\,302$. From \cite[Theorem 1.1]{ca2017}, it follows that
\begin{equation}
p_n - \frac{0.15p_n}{\log^3 p_n} < \vartheta(p_n) < p_n + \frac{0.15p_n}{\log^3 p_n}. \tag{4.2} \label{4.2}
\end{equation}
By Rosser and Schoenfeld \cite[Corollary 1]{rosserschoenfeld1962}, we have $m > p_m/\log p_m$ for every $m \geq 7$. Applying the last inequality to the 
left-hand side inequality of \eqref{4.2}, we get $\vartheta(p_n) > p_n - 0.15n/\log^2n$. Now we apply Theorem \ref{thm102} to get the desired lower bound for 
$\vartheta(p_n)$ for every $n \geq 841\,508\,302$. We check the remaining cases for $n$ with a computer.

Similarly to the first part of the proof, we apply the inequality $n > p_n/\log p_n$ to the right-hand side inequality of \eqref{4.2} to get $\vartheta(p_n) < 
p_n + 0.15n/\log^2n$. Now we use Theorem \ref{thm101} to get the required upper bound for $\vartheta(p_n)$ for every $n \geq 841\,508\,302$. For smaller values 
of $n$, we use a computer.
\end{proof}

\section{Appendix}

In the proof of Theorem \ref{thm102}, we note a table, which indicates the explicit values of $M_i(B_i)$ for given $B_i$.
In this appendix, we show that this table is indeed correct; i.e. $H_i(n) \geq 0$ for every positive integer $n \geq M_i(B_i)$ for the given values of $B_i$. 
We start with the claim concerning $H_1(n)$.

\begin{prop} \label{prop501}
We have $M_1(0.27) = 1\,359\,056\,314$.
\end{prop}

\begin{proof}
We have $Q_8(x) \geq 0$ for every $x \geq 0.6$ and $P_9(x) \geq 0$ for every $x \geq 0.6$. Using Lemma \ref{lem202}, we get
\begin{equation}
H_1(n) \geq \frac{f_1(w(n))}{4\log^6n\log p_n} \tag{5.1} \label{5.1}
\end{equation}
for every integer $n \geq 1\,338\,564\,587$, where $f_1(x) = 4 \cdot 0.27xe^{3x} - 2Q_7(x)e^x + 2 \cdot 0.87Q_8(x) + Q_9(x) + 2 \cdot 12.85 \cdot 0.87^2P_9(x)$. 
We show that $f(x) \geq 0$ for every $x \geq 3.05$. For this, we set $g(x) = (116.64+87.48x)e^x + (-24.6x^4-322.1x^3-1137.1x^2-1265.98x-512.24)$. It is 
easy to show that $g(x) \geq 212$ for every $x \geq 1.7$. So, $f_1^{(4)}(x) = g(x)e^x + 240x - 1005.6 \geq 212e^x + 240x - 1034.688 \geq 0$ for every $x \geq 
1.7$. Now, it is easy to see that $f(x) \geq 0$ for every $x \geq 3.05$. Applying this to \eqref{5.1}, we get $H_1(n) \geq 0$ for every integer $n \geq 
\exp(\exp(3.05))$. Finally, it suffices to verify the remaining cases with a computer.
\end{proof}

Before we check that $M_2(4.23) = 1\,471\,247\,583$, we introduce the following function.

\begin{defi}
For $x \geq 1$, let
\begin{displaymath}
\Phi(x) = e^x + x + \log \left( 1 + \frac{x-1}{e^x} + \frac{x-2.1}{e^{2x}} \right).
\end{displaymath}
\end{defi}

We note the following three properties of the function $\Phi(x)$.

\begin{lem} \label{lem502}
For every $x \geq 1$, we have $\Phi'(x) \geq e^x + 3/4$.
\end{lem}

\begin{proof}
We have $\Phi'(x) \geq e^x + 3/4$ if and only if $g(x) = e^{2x} - 3xe^x + 7e^x - 7x + 18.7 \geq 0$. Since $g''(x) = 4e^{2x} - (3x-1)e^x \geq 0$ for every $x 
\geq 0$ and $g'(1) \geq 10.49$, we obtain $g'(x) \geq 0$ for every $x \geq 1$. If we combine this with $g(1) \geq 29.96$, we get $g(x) \geq 0$ for every $x 
\geq 1$. 
\end{proof}

\begin{lem} \label{lem503}
For every $x \geq 1.25$, we have $\Phi(x) \geq e^x + x$.
\end{lem}

\begin{proof}
The desired inequality holds if and only if $(x-1)e^x +  x - 2.1 \geq 0$. Since the last inequality holds for every $x \geq 1.25$, we arrived at the end of the 
proof.
\end{proof}

\begin{lem} \label{lem504}
For every integer $n \geq 3$, we have $\Phi( \log \log n) \leq \log p_n$.
\end{lem}

\begin{proof}
The claim follows directly from \cite[Proposition 5.16]{dusart2017}.
\end{proof}

Next, we use these properties to determine the value $M_2(4.23)$.

\begin{prop} \label{prop505}
We have $M_2(4.23) = 1\,471\,247\,583$.
\end{prop}

\begin{proof}
We set $f_2(x) = 4.23x \Phi^3(x) + 12.85xe^x \Phi^2(x) - 71.3e^{3x}$ and use Lemmata \ref{lem502} and \ref{lem503} to obtain
\begin{equation}
f_2'(x) \geq 4.23(e^x + x)^3 + 25.54xe^x(e^x+x)^2 + 12.85e^x(e^x+x)^2 + 25.7xe^{2x}(e^x+x)-213.9e^{3x} \tag{5.2} \label{5.2}
\end{equation}
for every $x \geq 1.25$. We denote the right-hand side of the last inequality by $g_2(x)$. A straightforward calculation gives $g_2^{(3)}(x) \geq (1383.48x - 
3930.66)e^{3x} \geq 0$ for every $x \geq 2.85$. Now it is easy to see that $g_2(x) \geq 0$ for every $x \geq 3.02$. Applying this to \eqref{5.2}, we see that 
$f_2'(x) \geq 0$ for every $x \geq 3.02$. Since $f_2(3.05) \geq 16.797$, we obtain $f_2(\log \log n) \geq 0$ for every integer $n \geq \exp(\exp(3.05))$. 
Finally, we apply Lemma \ref{lem504}. For smaller values of $n$, we use a computer.
\end{proof}

\begin{prop} \label{prop506}
We have $M_3(1.575) = 1\,468\,111\,666$.
\end{prop}

\begin{proof}
Let $f_3(x) = 3.15 x \Phi(x) - 35.15x^2 + 44.6x - 42.08$. Using Lemmata \ref{lem502} and \ref{lem503}, we get $f_3'(x) \geq (3.15e^x + 3.15 - 67.15)x \geq 0$ 
holds for every $x \geq 3.02$. Combined with $f_3(3.05) \geq 0.044$ and Lemma \ref{lem504}, we get that $H_3(n) \geq 0$ for every integer $n \geq 
\exp(\exp(3.05))$. We conclude by a computer check.
\end{proof}

\begin{prop} \label{prop507}
We have $M_4(0.058) = 1\,383\,728\,153$.
\end{prop}

\begin{proof}
We set $f_4(x) = 0.116xe^x \Phi(x) + 3.15x^3 - 57.45x^2 + 113.01x - 80.05$ and have $f_4(3.05) \geq 0.812$. By Lemmata \ref{lem502} and \ref{lem503}, we get
$f_4'(x) \geq (0.116(e^x(e^x+x) + e^{2x}) + 9.45x - 114.9)x \geq 0$ for every $x \geq 2.92$. Hence $f_4(\log \log n) \geq 0$ for every integer $n \geq 
\exp(\exp(3.05))$. Finally it suffices to apply Lemma \ref{lem504}. For smaller values of $n$, we check the required inequality with a computer.
\end{proof}

\begin{prop} \label{prop508}
We have $M_5(2.24) = 1\,462\,324\,835$.
\end{prop}

\begin{proof}
To proof the claim, we define $f_5(x) = 4.48xe^x - 2x^4 + 5x^3 - 37.7x^2 + 41.1x - 31.9$. Since $f_5^{(3)}(x) \geq 0$ for every $x \geq 2.1$ and $f_5''(2.1) 
\geq 31.756$, we obtain $f_5''(x) \geq 0$ for every $x \geq 2.1$. Together with $f_5'(2.4) \geq 3.853$, we get $f_5'(x) \geq 0$ for every $x \geq 2.4$. Combined 
with $f_5(3.05) \geq 0.06$, we conclude that $f_5(\log \log n) \geq 0$, and thus $H_5(n) \geq 0$, for every integer $n \geq \exp(\exp(3.05))$. For smaller 
values of $n$, we verify the inequality $H_5(n) \geq 0$ with a computer.
\end{proof}

Adding the constants $B_1, \ldots, B_5$ given in Proposition \ref{prop501} and Propositions \ref{prop505}-\ref{prop508}, we get $12.85 - B_1 - B_2 - B_3 - 
B_4 - B_5 = 4.477$. Now we set $B_6 = 0.12$ to obtain the following explicit value for $M_6(B_6)$.

\begin{prop} \label{prop509}
We have $M_6(0.105) = 5$.
\end{prop}

\begin{proof}
Let $r(x,t) = (0.105e^x + 4.477)x \Phi(x) + 3.15xe^x - 3.15(t^2+1)e^t$, let $f_6(x) = r(x,x)$. If $t_0 \leq x \leq t_1$, then $f_6(x) \geq r(t_0,t_1)$. We 
check 
with a computer that $r(0.7 + i \cdot 10^{-3}, 0.7 + (i+1) \cdot 10^{-3}) \geq 0$ for every integer $i$ such that $0 \leq i \leq 2\,799$. Hence $f_6(x) \geq 0$ 
for every $x$ such that $0.7 \leq x \leq 3.5$. To show that $f_6(x) \geq 0$ for every $x \geq 3.5$, we set
\begin{displaymath}
g(x) = (0.105xe^x + 0.105e^x + 4.477)(e^x+x) + (0.105e^x + 4.477)xe^x - 3.15xe^x(1+x).
\end{displaymath}
Then $g'(x) = h(x)e^x + 4.477$ where $h(x) = 0.42(1+x)e^x - 3.045x^2 - 4.658x + 5.909$. Since $h(x) \geq 0$ for every $x \geq 3.09$, we get $g'(x) \geq 0$ for 
every $x \geq 3.09$. Together with $g(3.47) \geq 0$, we see that $g(x) \geq 0$ for every $x \geq 3.47$. Using Lemmata \ref{lem502} and \ref{lem503}, we obtain 
$f_6'(x) \geq g(x) \geq 0$ for every $x \geq 3.47$. Combined with $f_6(3.5) \geq 4.35411$, we have $f_6(x) \geq 0$ for every $x \geq 3.5$.  Hence $f_6(x) \geq 
0$ for every $x \geq 0.7$. Now we apply Lemma \ref{lem504} to get $H_6(n) \geq 0$ for every integer $n \geq \exp(\exp(0.7))$. We conclude by direct 
computation.
\end{proof}

\begin{prop} \label{prop510}
We have $M_7(0.0026) = 1\,075\,859\,481$.
\end{prop}

\begin{proof}
Substituting the definition of $P_8(x)$, we get
\begin{displaymath}
H_7(n) = \frac{0.0026w}{y^2z} - \frac{38.55w^2 - 77.1w + 66.82}{2y^3z^3}.
\end{displaymath}
To show that $H_7(n) \geq 0$ for every integer $n \geq 1\,075\,859\,481$, we first consider the function $f_7(x) = 0.0052xe^x \Phi^2(x) - 38.55x^2 + 77.1x - 
66.82$. We have $f_7(3.05) \geq 6.821$. Additionally, we use Lemmata \ref{lem502} and \ref{lem503} to get $f_7'(x) \geq (0.0052(e^x+x)^2(1+e^x) + 
0.0104e^{2x}(e^x + x) - 77.1)x \geq 0$ for every $x \geq 2.76$. Hence, $f_7(\log \log n) \geq 0$ for every integer $n \geq \exp(\exp(3.05))$. Finally, we can 
apply Lemma \ref{lem504}. For the remaining cases, we use a computer.
\end{proof}


\begin{prop} \label{prop511}
We have $M_8(0.052) = 1\,445\,815\,789$.
\end{prop}

\begin{proof}
We set $f_8(x) = 0.052x\, \Phi^2(x) - 12.85(x^2-x+1)$. We have $f_8(3.05) \geq 0.148$. By Lemmata \ref{lem502} and \ref{lem503}, we obtain $f_8'(x) \geq 
(0.052(e^x + x) + 0.104 (e^x + x)e^x - 25.7)x \geq 0$ for every $x \geq 2.66$. Hence $f_8(\log \log n) \geq 0$ for every integer $n \geq \exp(\exp(3.05))$. Now 
we use Lemma \ref{lem504} to obtain $H_8(n) \geq 0$ for every integer $n \geq \exp(\exp(3.05))$. For smaller values of $n$, we use a computer.
\end{proof}

\begin{prop} \label{prop512}
We have $M_9(0.1955) = 1\,479\,240\,488$.
\end{prop}

\begin{proof}
We define $f_9(x) = 0.1955x\,\Phi^4(x) - 463.2275e^{2x}$. By Lemmata \ref{lem502} and \ref{lem503}, we have $f_9'(x) \geq (0.1955(e^x+x)^2 + 0.782x(e^x+x)^2 - 
926.455)e^{2x} \geq 0$ for every $x \geq 2.83$. Combined with $f_9(3.05) \geq 7.11$, we get $f_9(x) \geq 0$ for every $x \geq 3.05$. Substituting $x = \log \log 
n$ in $f_9(x)$, we apply Lemma \ref{lem504} to see that $H_9(n) \geq 0$ for every integer $n \geq \exp(\exp(3.05))$. For every integer $n$ such that 
$1\,479\,240\,488 \leq n \leq \exp(\exp(3.05))$ we check the desired inequality with a computer.
\end{proof}

Finally, we determine the value of $M_{10}(0.08)$.

\begin{prop} \label{prop513}
We have $M_{10}(0.08) = 1\,447\,605\,594$.
\end{prop}

\begin{proof}
Let $f_{10}(x) = 0.08x \,\Phi^5(x) - 4585e^{2x}$. Applying Lemmata \ref{lem502} and \ref{lem503}, we get $f_{10}'(x) \geq (0.4x(e^{x}+x)^3 - 9170)e^{2x} 
\geq 0$ for every $x \geq 2.9$. Together with $f_{10}(3.05) \geq 6142.27$, we see that $f_{10}(\log \log n) \geq 0$ for every integer $n \geq 
\exp(\exp(3.05))$. Now, we use Lemma \ref{lem504} to conclude that $H_{10}(n) \geq 0$ for every integer $n \geq \exp(\exp(3.05))$. Finally, it suffices 
to verify the remaining cases with a computer.
\end{proof}


\end{document}